\documentclass{aptpub1}
\usepackage{xcolor}
\authornames{Y.~Malinovsky and Y.~Rinott} % insert the authors here for use in running head. If three or more authors please use (for example) M.~YARROW {\it et al}. Author names should follow the same M.~YARROW format and if two authors, separate by 'AND'.
\shorttitle{Tournaments and Negative Dependence} % insert short title here for use in running head

\begin{document}%\recd{}{}%Do not alter this line.

\title{On Tournaments and Negative Dependence}

\authorone[University of Maryland, Baltimore County]{Yaakov Malinovsky}
\authortwo[The Hebrew University of Jerusalem]{Yosef Rinott}

\addressone{Department of Mathematics and Statistics, University of Maryland, Baltimore County, Baltimore, USA.} % Your postal address goes here.
\emailone{yaakovm@umbc.edu} %Authors email goes here.
\addresstwo{Department of Statistics and Federmann Center for the Study of Rationality, The Hebrew University of Jerusalem, Israel.}
\emailtwo{yosef.rinott@mail.huji.ac.il}

\begin{abstract}
Negative dependence of sequences of random  variables is often an interesting characteristic of their distribution, as well as a useful tool for studying various asymptotic results, including central limit theorems, Poisson approximations, the rate of increase of the maximum, and more. In the study of probability models of tournaments,
negative dependence of  participants' outcomes arises naturally with application to various asymptotic results.
In particular, the property of \textit{negative orthant dependence} was proved in several articles  for different tournament models, with a special proof for each model. In this note we unify these results by proving a stronger property, \textit{negative association}, a generalization leading to a very simple proof. We also present a natural  example of a knockout  tournament where the scores are negatively orthant dependent but not negatively associated. The proof requires a new  result on a preservation property of negative orthant dependence that is of independent interest.
\end{abstract}
\smallskip

%insert keywords separated by a semicolon. You should avoid including keywords which also appear in the title.

% insert the primary 2020 Maths Subject Classification number in the first bracket
		% and the secondary ams number(s) in the second bracket
		% e.g. \ams{60E20}{49G03;49F10}
		%Maximum of three in each, ideally one or two in each primary and secondary.
		%codes found here ``https://mathscinet.ams.org/msnhtml/msc2020.pdf''

\bigskip

\section{Introduction}\label{sec:int}
\subsection{Tournaments}
A tournament consists of competitions between several players where the final score or payoff of each player is determined by the sum of scores of the player's matches. For a tournament with $n$ players, let
${\bf S}=(S_1,\ldots,S_n)$ denote the vector of their final scores.   Under natural probability models and in many kinds of tournaments, the components of ${\bf S}$ exhibit some type of negative dependence. We briefly define two concepts of dependence to be considered in this paper and then we discuss various tournaments where these concepts are relevant. We present a theorem on negative association that unifies and strengthens known results on negative dependence of tournament scores, and leads to new ones. Specifically, we prove negative association in various models.  We also analyze a tournament in which, interestingly,  negative association holds when the draw of matches is random, and otherwise only a weaker notion of negative dependence, negative orthant dependence, holds.

\subsection{Two notions of negative dependence}
We define the following negative dependence notions. See \cite{JP1983} and references therein for details.
Throughout this paper increasing (decreasing) stands for nondecreasing (nonincreasing).
\begin{definition}(\cite{JP1983}, Definition 2.3)\label{def:NOD}
	The random variables $S_1,...,S_n$ or the vector ${\bf S}=(S_1,\ldots,S_n)$ are said to be \textit{negatively lower orthant dependent} (NLOD) if for all $s_1,\ldots,s_n$ $\in \mathbb{R}$,
	\begin{equation}\label{eq:L}
%	\label{eq:iH}
	P\left(S_1 \le s_1,\ldots,S_n \le s_n\right)\leq P\left(S_1 \le s_1\right)\cdots P\left(S_n \le s_n\right),
	\end{equation}
	and \textit{negatively upper orthant dependent} (NUOD) if
	\begin{equation}\label{eq:U}
	P\left(S_1>s_1,\ldots,S_n>s_n\right)\leq P\left(S_1>s_1\right)\cdots P\left(S_n>s_n\right).
	\end{equation}
	\textit{Negative orthant dependence} (NOD) is said to hold if both \eqref{eq:L} and \eqref{eq:U} hold.
	\end{definition}

\begin{definition}(\cite{JP1983}, Definition 2.1)\label{def:NA}
	The random variables $S_1,\ldots,S_n$ or the vector ${\bf S}=(S_1,\ldots,S_n)$ are said to be {\it negatively associated} (NA) if for every pair of disjoint subsets
	$A_1, A_2$ of $\left\{1,2,\ldots,n\right\}$,
	\begin{equation}\label{eq:cov}
	Cov\left(f_1(S_i, i\in A_1), f_2(S_j, j\in A_2)\right)\leq 0,
	\end{equation}
	whenever $f_1$ and $f_2$ are real-valued functions,  increasing  in all coordinates.
\end{definition}
Clearly NA implies NOD  (see \cite{JP1983}). In Section \ref{sec:nonRknok} we provide a natural example of a tournament where $S_1,\ldots,S_n$ are NOD but not NA.

\subsection{Motivation}
\subsubsection{{\bf General dependence structure}}
The study of dependence structure between random variables and related stochastic orders is of interest in itself; see, e.g., the books \cite{MS2002, J1997, J2015}, and articles such as \cite{P2000} and \cite{D2016} which concentrate on negative dependence and its applications.
Dependence models are relevant to a large number of applications, such as system reliability and risk theory \cite{MS2002, R2013}, statistical mechanics \cite{R1999}, asymptotic approximations \cite{BS2007} and non-asymptotic bounds on the difference between certain functions of dependent random variables, and simple models with independence \cite{BHJ1992,GW2018}, in multiple testing hypotheses \cite{SC1996,RP1980,SC1997,BY2001}, various optimization problems; see, e.g., \cite{PW2015}, and geometric probability \cite{N1984}.

\subsubsection{{\bf Negative Dependence and Tornaments }}
Negative association (NA) and other concepts of negative  dependence are relevant to tournaments, as explained below. In the present  paper we unify results which appear in the literature on tournaments, and extend them to the strong notion of NA, and to general classes of tournaments.

Pemantle \cite{P2000} states that {\it "the property of NA is reasonably useful but hard to verify."} We provide simple tools and examples where NA is verified in the context of tournaments.

Huber \cite{H1963} considered a certain tournament model (details provided in the next section) where player 1 is stronger than all other players, who are all the equally strong. He proved that  $\lim_{n \rightarrow \infty}P(S_1>\max\{S_2,\ldots, S_n\}) \rightarrow 1$, that is, Player 1 achieves the highest score with probability approaching 1. His proof is based on the fact (which he proves by a special coupling argument) that  the components of $\bf S$ are NLOD; we give a simpler proof showing the stronger property of NA.
Ross \cite{R2021} studied a binomial tournament model (details provided in the next section) and established bounds for $P(S_i> \max_{j\neq i} S_j)$ using stochastic ordering property which required the knowledge
of certain negative dependence structure of the scores (see also \cite{R2016}).
Malinovsky and Moon \cite {MM2021} studied convergence in probability of the normalized maximal score  to a constant for a general tournament model (details are given in the next section) by using inequalities for the joint distribution function of the scores $S_1,\ldots, S_n$; the proof  requires the  NLOD property of the scores.
Malinovsky \cite{M2021a, M2021b} established the asymptotic distribution of the maximal score, second maximal, etc., in a chess round-robin tournament model (details provided in the next section) using a non-asymptotic bound on the total variation distance between the sum of indicators that the score of player $j$ is larger than a given constant and a suitable Poisson approximation which would hold if the indicators were independent. This bounds is based on the fact that the indicators
have certain negative dependence structure.
It follows that one can use classical limiting results under independence and show that the maximal score and related functional have Gumbel-type distribution in the limit.  In all these examples we provide a simple proof of NA which implies the required negative dependence and in the last example our proof holds for a complete range of the parameters, unlike the proof in \cite{M2021b}. Thus we unify and simplify many existing results in the literature, extending the range of tournament models and and strengthening the dependence proved.

\subsection{Constant-sum round-robin tournaments}\label{sec:RR}
%{\textbf{A general round-robin tournament}}\label{sec:grrt}
We start with a formulation of a \textit{general constant-sum round-robin tournaments}. See, e.g., \cite{Moon2013} and \cite{BF2018}. Assume that each of $n$ players competes against each of the other $n-1$ players. When player  $i$ plays against $j$, where $i<j$, player $i$'s reward is a random variable $X_{ij}$ having a distribution function $F_{ij}$ with support on $[0, r_{ij}]$, and $X_{ji}=r_{ij}-X_{ij}$; for $i<j$ this determines $F_{ji}(t)=1-F_{ij}(r_{ij}-t)$ for
$t \in [0, r_{ij}]$. Thus each pair of players competes for a share of a given reward. We assume that $X_{ij}$ are independent for $i<j$, and also that $r_{ij} \ge 0$.
The case where $r_{ij} = 0$ has the interpretation that players $i$ and $j$  do not compete against each other.
 The total reward for player $i$ is defined for  all tournaments we consider by $$S_i=\sum_{j=1, j\neq i}^{n}{X_{ij}},\quad i=1,\ldots,n.$$ The sum of the rewards is constant: $\sum_{i=1}^n S_i= \sum_{i<j}r_{ij}$.

 We shall prove that $S_1,\ldots,S_n$ are NA (Definition \ref{def:NA}), extending and simplifying various results in the literature, to be specified  below, and
more generally, if $u_i$ are increasing functions, it follows that that $u_1(S_1),\ldots, u_n(S_n)$ are also NA. These functions can represent the utilities of the players. See Proposition \ref{propgenrr} for a further generalization.

  \subsubsection{\textbf{A round-robin tournament with integer reward }}
  The case of the above round-robin tournament model with an integer support  $\{0,1,\ldots,r_{ij}\}$ of $F_{ij}$ was considered recently in Malinovsky and Moon \cite{MM2021}. Our results on negative dependence for the general round-robin tournament generalize the negative dependence results in
 \cite{MM2021}. Specifically, the NLOD property is  proved in \cite{MM2021}, and our general result yields the NA property with a simpler proof.

  We next discuss further special cases of our general formulation that have appeared in the literature.

 \subsubsection{\textbf{A round-robin tournament with pairwise repeated games}}
 %\textbf{Pairwise repeated games:}
 Recently, Ross \cite{R2021} considered a special case of  the above two models where $X_{ij}\sim Binomial(r_{ij}, p_{ij})$ independently for all $i<j$,  $r_{ij}=r_{ji}$ and $X_{ji}=r_{ij}-X_{ij} \sim Binomial(r_{ij},1-p_{ij})$.
 As always, $S_i=\sum_{j=1, j\neq i}^{n}{X_{ij}}$. This model arises if each pair of players $(i,j)$ plays $r_{ij}$ independent games, and $i$ wins with probability $p_{ij}$.   Ross \cite{R2021} obtained NOD-type  results for general $p_{ij}$ using log-concavity, conditioning, and Efron's well-known theorem \cite{EF1965}.
 Again we strengthen and simplify these results
 and prove the NA property.  Ross  used his results
 to study expressions such as $P(S_i> \max_{j\neq i} S_j)$ and related inequalities, under a special model for $p_{ij}$, given, e.g., in
 \cite{Z1929, BT1952}.

 \subsubsection{\textbf{A simple round-robin tournament}} Huber \cite{H1963} considered the above general model where for any $i \neq j$, $X_{ij}+X_{ji}=1, X_{ij}\in \left\{0,1\right\}$ and $P(X_{ij}=1)=p_{ij}$, and proved that $S_1,\ldots,S_n$ are NLOD by invoking coupling arguments.  He used the latter fact to prove that if $P(X_{1j}=1)=p>1/2$, and $P(X_{ij}=1)=1/2$ for all $1<i \neq j \le n$,  then $\lim_{n \rightarrow \infty}P(S_1>\max\{S_2,\ldots,S_n\}) \rightarrow 1$; that is, Player 1 achieves the highest score with probability approaching 1.

\subsubsection{\textbf{A chess round-robin tournament with draws}}
\label{sec:chesfoot}
Malinovsky \cite{M2021a} and \cite{M2021b} considered the following round-robin tournament model:
for $i\neq j$, $X_{ij}+X_{ji}=1,\,\,X_{ij}\in \left\{0,\, 1/2,\, 1\right\}$;  this can be seen as a special case of the general model where $F_{ij}$ have the support $\{0,\, 1/2,\, 1\}$.
Malinovsky  considered the case where all players are equally strong, i.e. $P\left(X_{ij}=1\right)=P\left(X_{ji}=1\right)$, and where the probability of a draw, $p=P\left(X_{ij}=1/2\right)$ is common to all games.
He proved
a  type of negative dependence called \textit{negative relation} which is weaker than NA (see \cite{BHJ1992}, Chapter 2)
for $S_1,\ldots,S_n$ using log-concavity
of the probability function of $2X_{ij}$, which  requires restricting the range of $p$ to $p=0$ or $p\in [1/3,1)$. He then used  results from \cite{BHJ1992} to prove a Poisson approximation to the number of times $S_i$ exceeds a certain threshold. We strengthen his result to NA, which in fact holds for all $p$ and more generally for all values of $P\left(X_{ij}=1\right)$ and $P\left(X_{ij}=1/2\right)$, that is, the above assumptions of equality of strength and a common probability of draw are dropped.

\subsection{\textbf{Random-sum n-player games}}\label{sec:RSum}
%\textbf{Fixed-sum n-player games:}
The following somewhat abstract description of a tournament is a generalization of all the above tournament models. Consider a sequence of $K$ $n$-player  games (rounds), where the random payoff to player $i \in \{1,\ldots,n\}$ in round $k \in \{1,\ldots,K\}$ is $X_{i}^{(k)}$ and the components of each of the payoff vectors  ${\bf X}^{(k)}=(X_{1}^{(k)},\ldots,X_{n}^{(k)})$ are NA, with ${\bf X}^{(k)}$'s being independent. In general, the sum of the components of each ${\bf X}^{(k)}$ is assumed to be a random variable. Constant-sum (or, equivalently, zero-sum) examples are formed  when the payoff vectors ${\bf X}^{(k)}$ have the multinomial or Dirichlet distribution (see \cite{JP1983} Section 3.1 for these and further examples).
An example where the sum of the players' payoffs in each game is random is the case where the vector ${ \bf X}^{(k)}$ is jointly normal with correlations $\le 0$ (\cite{JP1983} Section 3.4).

The total payoff to player $i$ in the $K$ rounds is $S_i=\sum_{k=1}^K X_{i}^{(k)}$,\,\,$i=1,\ldots,n$. We shall prove in Section \ref{sec:NNAARR} Theorem \ref{prop:rrsum} that $S_1,\ldots,S_n$ are NA.
More generally, one can take $S_i=u_i\left(X_{i}^{(1)},\ldots,X_{i}^{(K)}\right)$ where  $u_i$ is any increasing function of player $i$'s payoffs.
Note that here, unlike in pairwise duels, several and even  all players may compete in each round. The  limiting distribution of the number of pure Nash equilibria in such random games was studied in Rinott and Scarsini
\cite{RS2000}.

\noindent{\textbf{Two sport examples}}
 A football league (Soccer in the US) provides an example of a random-sum round-robin tournament. The winning team is awarded three points,  and if the game ends in a tie, each team receives one point. For a single match the score possibilities for the two teams are $(3, 0), (1, 1)$ and $(0,3)$ with some probabilities, forming an NA distribution for any probabilities. Let the $n$-dimensional vectors ${ \bf X}^{(\bf k)}$, for ${\bf k}=(ij)$ with $i \neq j$, consist of zeros except for two coordinates $i$ and $j$ corresponding to
the playing teams  $i$ and $j$,
where one of the above three vectors appears.  Then ${\bf S} =\sum_{{\bf k}=(i,j): 1\le i <j \le n}{\bf X}^{{(\bf k)}}$ represents the vector of total scores of the $n$ teams after they all play each other.   It is easy to see that each vector ${ \bf X}^{(\bf k)}$ is NA.
Equivalently one can assume that ${ \bf X}^{(\bf k)}$ contains the scores of all players in all matches in week $k$.

Under some assumption (which are an approximation to reality), the Association of Tennis Professionals (ATP) ranking is another example. It can be seen as a tournament in which the number of points awarded to the winner of each game depends on the tournament and the stage reached. Players' ranks are increasing functions of their total scores. Here we do not assume that each player plays against all others in the ATP ranking, which is expressed by setting some of the rewards to be zero.

\subsection{Knockout tournaments}\label{sec:KNT}
Consider a knockout tournament with $n=2^\ell$ players of equal strength; that is, player $i$ defeats player $j$ independently of all other duels with  probability 1/2 for all $1 \le i \neq j \le n$. The winner continues to a duel with another winner, and the defeated player is eliminated from the tournament.
Let $S_i$ denote the number of games won by player $i$. We could also replace $S_i$ by the prize money of player $i$, which in professional tournaments is usually an increasing function of $S_i$. For a completely random schedule of matches (aka the draw; see \cite{ACKPR2017}),
we show in Section  \ref{sec:knock} that the vector $\textbf{S}=(S_1,\ldots,S_n)$ is NA. Note that in tennis tournaments such as Wimbledon the draw is not completely random as top-seeded players' matches are drawn in a way that prevents them from playing against other top-seeded players in early rounds.   For non-random draws we prove the NOD property via a new preservation result, and we provide a counterexample to the NA property; thus it need not hold for fixed, non-random draws. We also provide an example where  NOD and NA do not hold if players are not of equal strength.

\section{Negative association and round-robin tournaments}\label{sec:NNAARR}
The following theorem generalizes Application 3.2(c) of \cite{JP1983}; it implies that the scores $S_1,\ldots, S_n$
in the general round-robin model of Section \ref{sec:RR} and therefore in all round-robin models of Section \ref{sec:RR} are NA, and therefore NLOD and NUOD, and NOD.

\begin{theorem}\label{th:propro}
	Let $X_1,\ldots,X_n$ be independent random variables  and let $g_i, i=1,2,\ldots,n$ be decreasing functions. Set $Y_1=g_1(X_1), \ldots,Y_n=g_n(X_n)$, and for $j=1,\ldots,m$ set
	$$
	S_j=f_j\left(\left\{X_i: i\in A_j\right\},\left\{Y_i: i\in B_j\right\} \right),
	$$
	where $f_j$ are coordinate-wise increasing functions of $|A_j|+|B_j|$ variables, and the sets $A_1,\ldots,A_m$ are disjoint subsets of $\left\{1,2,\ldots,n\right\}$, and so are $B_1,\ldots,B_m$. Then the random variables $S_1,\ldots,S_m$ are NA.
\end{theorem}

\begin{proof} The pair of variables $X,g(X)$ with $g$ decreasing is NA. This is well known; for completeness, here is a simple proof. Let $X^*$ be an independent copy of $X$. For increasing functions $f_1$ and $f_2$, we have
	$$2Cov(f_1(X), f_2(g(X)\,)=
	E\{[f_1(X)-f_1(X^*)\,][f_2(g(X))-f_2(g(X^*))\,]\}
	\le 0\,, $$
	since the expression in the expectation is $\le 0$.
	The pairs $(X_1, Y_1), \ldots,(X_n, Y_n)$ are independent and  each pair is NA. Property $P_7$ of \cite{JP1983} states that the union of independent sets of NA random variables is NA.
	Therefore the random variables $X_1,\ldots,X_n, Y_1, \ldots,Y_n$ are NA.
	Property $P_6$ in \cite{JP1983} states that increasing functions defined on disjoint subsets of a set of NA random variables are NA. Therefore   $S_1,\ldots,S_m$ are NA.
\end{proof}
We now apply Theorem \ref{th:propro} to show the NA property in the general round-robin model of Section \ref{sec:RR}.

\begin{proposition} \label{propgenrr} Let $X_{ij} \sim F_{ij}$ with support on $[0, r_{ij}]$ be independent  for $1 \le i<j \le n$, where $r_{ij}\ge 0$, and let $X_{ji}=r_{ij}-X_{ij}$.
 Set $S_i=\sum_{j=1, j\neq i}^{n}{X_{ij}},\,\,\, i=1,\ldots,n$. Then $S_1,\ldots,S_n$ are NA.
 More generally, if we set $S_i=u_i(X_{i1},\ldots,X_{i,i-1},X_{i,i+1},\ldots,X_{in})$, \,$i=1,\ldots,n$, where the $u_i$'s are any increasing functions, we again have that the variables $S_1,\ldots,S_n$ are NA.
  \end{proposition}

\begin{proof}
	Instead of a single index we apply Theorem \ref{th:propro} to the independent doubly indexed random variables $X_{ij}$ for $i<j$. Let $g_{ij}(x)=r_{ij}-x$, so that $X_{ji}=g_{ij}(X_{ij})=r_{ij}-X_{ij}$ with $X_{ji}$ playing the role of $Y$'s in Theorem \ref{th:propro}.
	Since the $S_i$'s are sums of disjoint subsets of the variables defined above, the result follows by Theorem \ref{th:propro}, and the same argument holds  with the  functions $u_i$  replacing the sums.
\end{proof}

Since all round-robin models of Section \ref{sec:RR}  are special cases of the general round robin model, we have:
\begin{corollary}
The NA property for $S_1,\ldots,S_n$ holds in all the round-robin models in Section \ref{sec:RR}. The NLOD results proved in the literature for these models follow; moreover, NUOD and hence NOD also follow.
\end{corollary}
The football example of Section \ref{sec:RSum} is not a special case of the constant-sum  general round-robin model; here the NA property follows by Theorem  \ref{th:propro} by replacing the  functions $g_{ij}$ by $g$ defined by $g(3)=0, g(1)=1$, and $g(0)=3$. It also follows by Theorem \ref{prop:rrsum} below.

We now consider the random-sum $n$-player games tournament of Section \ref{sec:RSum}.

\begin{theorem}\label{prop:rrsum}
Consider the random-sum tournament model of Section \ref{sec:RSum}, that is, a sequence of $K$ $n$-player  games (rounds), where the random payoff to player $i \in \{1,\ldots,n\}$ in round $k \in \{1,\ldots,K\}$ is $X_{i}^{(k)}$ and the components of each payoff vector  ${\bf X}^{(k)}=\left(X_{1}^{(k)},\ldots,X_{n}^{(k)}\right)$ are NA. The vectors $\textbf{X}^{(k)}$ are distributed independently. Let $S_i=\sum_{k=1}^K X_{i}^{(k)}$.
Then $S_1,\ldots,S_n$ are NA. More generally, the variables  $S_i=u_i\left(X_{i}^{(k)},\ldots,X_{i}^{(K)}\right)$, \,$i=1,\ldots,n$, where the $u_i$'s are any increasing functions, are NA.
\end{theorem}

The above theorem can be restated in the following lemma, which follows readily from properties of negative association given in \cite{JP1983}.
The same result for positive association, with the same proof, is given in \cite{KR1980} Remark 4.2.
\begin{lemma}\label{le:conNA}
	The convolution of NA vectors is NA.
	\end{lemma}
\begin{proof}
	Let ${\bf X}^{(k)} \in \mathbb{R}^n$ be independent NA vectors and let ${\bf S}=(S_1,\ldots,S_n)=\sum_{k=1}^K \textbf{X}^{(k)}$.  By Properties $P_7$ and then $P_6$ of \cite{JP1983}, the union of all variables in these vectors is NA and hence $S_1,\ldots,S_n$ are NA since they are increasing functions of disjoint subsets of the above union.
	\end{proof}
The above argument holds also when $S_i=u_i\left(X_{i}^{(1)},\ldots,X_{i}^{(K)}\right)$, thus proving the last part of Theorem \ref{prop:rrsum}.

The next corollary shows that the NA property of the general round-robin model of Section \ref{sec:RR} and hence in all the models of \ref{sec:RR} follows also from Theorem \ref{prop:rrsum}.

\begin{corollary}\label{cor:RSNA}
	The scores $S_1,\ldots,S_n$ of the general round-robin models in Section \ref{sec:RR} are NA.
\end{corollary}

\begin{proof} For clarity we start with the simple case of $n=3$. Define the vectors ${\bf Y}^{12}=(X_{12}, r_{12}-X_{12}, 0),\,\, {\bf Y}^{13}=(X_{13}, 0, r_{13}-X_{13})$, and  ${\bf Y}^{23}=(0, X_{23}, r_{23}-X_{23})$ with $X_{ij}$ of the general round-robin model. It is easy to see that $S_i=\sum_{1 \le k<\ell \le 3} Y^{kl}_i$.
   	
In general, starting with the rewards $X_{ij}$ of the general round-robin model, form  the $K=:n(n-1)/2$ vectors  ${\bf Y}^{ij} \in \mathbb{R}^n$,\,$1 \le i < j \le n$,  whose $i$th component, $Y^{ij}_i$, equals $X_{ij}$, its $j$th component,  $Y^{ij}_j$, equals $r_{ij}-X_{ij}$, and the remaining components equal zero. The components $(Y^{ij}_1,\ldots, Y^{ij}_n)$ of each of the $K$ vectors ${\bf Y}^{ij}$ are obviously NA. Setting
 $$S_i=\sum_{1 \le k < \ell \le n} Y^{kl}_i, \,\,i=1,\ldots,n,$$
 it is easy to see that these $S_i$ coincide with those of the general round-robin model.
 Theorem \ref{prop:rrsum} applied to the K vectors ${\bf Y}^{kl}$ implies that the variables $S_i$ are NA.
\end{proof}

\section{Knockout tournaments}\label{sec:knock}
We now discuss negative dependence in the knockout tournament of Section \ref{sec:KNT}.
\subsection{Knockout tournaments with a random draw}
\begin{proposition}
\label{prop:3}
Consider a 	knockout tournament starting with $n=2^\ell$ players, where player $i$ defeats player $j$ independently of all other duels with  probability 1/2 for all $1 \le i \neq j \le n$;  the winner continues to a duel with another winner, and the defeated player is eliminated from the tournament. Let $S_i$ denote the number of games won by player $i$. Assume a completely random schedule (draw) of the matches. Then $S_1,\ldots,S_n$ are NA.
\end{proposition}	
\begin{proof}
First note that for a given $\ell$,  the vector ${\bf S}=(S_1,\ldots,S_n)$ contains the components $i=0,\ldots, \ell$ with $i <\ell$ appearing $2^{\ell-1-i}$ times, and $\ell$ appearing once. For example, if $n=4$ ($\ell=2$) then there are 2 players with 0 wins, 1 player (the losing finalist) with 1 win, and 1 player  (the champion) with 2 wins. Thus, the vector $\textbf{S}$ is a permutation of the vector $(0,0,1,2)$.
If $n=8$  ($\ell=3$) then $\bf{S}$ is a permutation of the vector $(0,0,0,0,1,1,2,3)$.
Under the assumption of a random draw, all permutations are equally likely as all players play a symmetric role.
Theorem 2.11 of \cite{JP1983} states that if ${\bf X}=(X_1,\ldots,X_n)$ is a random permutation of a given list of real numbers, then $\textbf{X}$ is NA, and the result follows.
\end{proof}

Without the assumption that players have equal probabilities in each duel, negative association
as in Proposition \ref{prop:3} need not hold. To see this consider the case of 4 players and assume
first that the relations between the players are deterministic; specifically, Player 1 beats Player 2 with probability 1 and loses to 3 and 4 with probability 1. Player 2 beats 3 and 4 with probability 1, and Player 3 beats 4 with probability 1. These relations are not transitive (for example, Player 1 beats 2 who beats 3, but 3 beats 1), which is not uncommon in various sports.
With a random draw, the vector $S=\left(S_1, S_2, S_3, S_4\right)$ can only take the outcomes
$\left(1, 0, 2, 0\right)$ (when Player 1 meets 2 in the first round), $\left(0, 2, 1, 0\right)$ (1 meets 3 in the first round) and $\left(0, 2, 0, 1\right)$ (1 meets 4 in the first round), each with probabilities $1/3$.
Let $f_1(S_1)=S_1$ and $f_2(S_3)=S_3$. Then $Ef_1(S_1)f_2(S_3)=2/3>Ef_1(S_1)Ef_2(S_3)=1/3$,
whereas $Ef_1(S_1)=1/3$ and $Ef_2(S_3)=1$, which contradicts negative association. If we replace the
probabilities of 1 by $1-\varepsilon$ for small $\varepsilon$ then the same result holds by an obvious continuity argument, so deterministic relations are not necessary for this example.
In the above example the vector ${\bf S}$ is not even NLOD. In fact
$P(S_1\leq 0, S_3\leq 0)=1/3 > 2/3\cdot1/3=P(S_1\leq 0)P(S_3\leq 0).$

\subsection{Knockout tournaments with a non-random draw}\label{sec:nonRknok}
This section provides a counterexample showing  that for knockout tournaments with a given non-random draw, the scores $S_1,\ldots,S_n$ need not be NA; however, we prove that they are NOD. To obtain the latter result we prove a result on NOD (and NLOD and NUOD) of independent interest.

\noindent \textbf{A counterexample to NA}
 Consider a knockout tournament with $n=4$ players and a draw where in the first round player 1 plays against 2, and 3 against 4. In this case only 8 permutations of $(0,0,1,2)$ are possible and one of the first two coordinates must be positive and so $(0,0,1,2)$ itself is is not a possible outcome.  Consider the functions $f_1(S_1,S_3)$ taking the value 0 everywhere, except that $f_1(0,1)=f_1(0,2)=1$, and $f_2(S_2,S_4)$ which is 0 everywhere, except for $f_2(2,0)=1$. We have $E[f_1(S_1,S_3)f_2(S_2,S_4)]=1/8$, $Ef_1(S_1,S_3)=2/8$, and $Ef_2(S_2,S_4)=1/8$, and \eqref{eq:cov} does not hold. \hspace{11.3cm} $\square$\\
In a tennis tournament, the above arrangement of matches occurs if Players 1 and 3 are top-seeded and the draw prevents them from being matched against each other in the first round.

Finally, we prove that in a knockout
tournament with a non-random schedule, $\textbf{S}=(S_1,\ldots,S_n)$ is NOD. We need the following theorem, which may be of independent interest.
%\textcolor{red}{Yakov : can you find a better/another reference? is it in Muller and Stoyan which I don't have?}
\begin{theorem}\label{th:presNOD} Let ${\bf X}^{(k)} = \left(X^{(k)}_{1},\ldots,X_{n}^{(k)}\right)\in \mathbb{R}^n$, $k=1,\ldots,K$ satisfy the following two assumptions.
	\begin{itemize}
		\item[\rm{(\bf i)}]
		For all $k=1,\ldots,K$:\,\,	${\bf X}^{(k)}\mid{\bf X}^{(k-1)}+\ldots+{\bf X}^{(1)}$   is NLOD, and
		\item[\rm{(\bf ii)}]
		For all $k$ and $i$: ${X}_{i}^{(k)}\mid{\bf X}^{(k-1)}+\ldots+{\bf X}^{(1)}\,\,\stackrel{d}{=}\,\,{X}_{i}^{(k)}\mid{X}_{i}^{(k-1)}+\ldots+{X}_{i}^{(1)}$\,;
	\end{itemize}
	
\noindent	that is, the conditional distribution of ${X}_{i}^{(k)}$ depends only on the $i$th coordinate of the sum of its predecessors. Then ${\bf X}^{(1)}+\ldots+{\bf X}^{(K)}$ is NLOD. Moreover, the result holds if we replace NLOD by NUOD and hence also by NOD.
\end{theorem}
\begin{proof} It is well known that a random vector  ${\bf Z}=(Z_1,\ldots,Z_n)$ is NLOD if and only if
	$E\prod_{i=1}^n \phi_i(Z_i) \le \prod_{i=1}^n E\phi_i(Z_i)$ for any nonnegative decreasing functions $\phi_i$ (Theorem 6.G.1 (b) in \cite{SS2007} or Theorem 3.3.16 in \cite{MS2002}).  The proof proceeds by induction, and it is easy to see that it suffices to prove it for $K=2$. Set  ${\bf X}:={\bf X}^{(1)}$ and ${\bf Y}:={\bf X}^{(2)}$. We have
	$$E\prod_{i=1}^n \phi_i(X_i+Y_i)=E\{E[\,\prod_{i=1}^n \phi_i(X_i+Y_i)\, |\, {\bf X}\,]\}\le E\prod_{i=1}^n E[\phi_i(X_i+Y_i)\mid {\bf X}]= E[\prod_{i=1}^n g_i(X_i)],$$ where $g_i(X_i)=E[ \phi_i(X_i+Y_i)\mid {\bf X}]$, and the inequality holds by Assumption {\rm{(\bf i)}}. By {\rm{(\bf ii)}} we have that $g_i(X_i)$ indeed depends only on $X_i$, and it is obviously nonnegative and decreasing.  By the NLOD property of ${\bf X}$ we have  $$E\prod_{i=1}^n g_i(X_i) \le \prod_{i=1}^n Eg_i(X_i)=\prod_{i=1}^n E\phi_i(X_i+Y_i),$$
	and the result follows. The same proof holds for NUOD with the functions $\phi_i$ taken to be increasing.
\end{proof}

A special case of Theorem \ref{th:presNOD} is the following corollary that for nonnegative vectors follows from Theorem 6.G.19 of \cite{SS2007} and can be obtained from Theorem 1 of \cite{Leh1966} (for vectors in $\mathbb{R}^2$) and from Theorem 4.2 (e) of \cite{Mull19977}.
\begin{corollary}
The sum of independent NOD (NLOD, NUOD) vectors is NOD (NLOD, NUOD).
	\end{corollary}
\begin{proposition}
	For the knockout tournament with a non-random draw, the vector $\textbf{S}=(S_1,\ldots,S _n)$ is NOD.
\end{proposition}
\begin{proof}
	Without loss of generality assume that in the first round player $2i-1$ plays against $2i$ for $i=1,\ldots, n/2$. Let $X^{(1)}_{j}=0\,\, (1)$ if player $j$
	loses (wins) the first round, $j=1,\ldots,n$. The pairs of variables $X_{2i-1}^{(1)},X_{2i}^{(1)}$ are independent and NOD (in fact they are NA), taking the values $(0,1)$ or $(1,0)$
with probability 1/2. It follows readily that the 0-1 vector ${\bf X}^{(1)}=\left(X_{1}^{(1)},\ldots,X_{n}^{(1)}\right)$, whose $j$th coordinate indicates a win or a loss of player $j$ in the first round,  is NOD. Now the second round is similar with only half the players, those who won the first round, where the value 0 is set for players who lost in the first round. Continuing this way, we see that the vector $(S_1,\ldots,S_n)$ is the sum of the 0-1 vectors of all the rounds. These vectors are not independent because the value of 0 in a coordinate of a vector pertaining to a given round must by followed by a zero there in the next round. However, \rm{(\bf i)} and \rm{(\bf ii)} of Theorem \ref{th:presNOD} are easily seen to hold, and the NOD property follows.
\end{proof}

\ack % Place the text of your acknowledgements after the \acks (or \Acks) command. This will generate the heading "Acknowledgements". If you wish to make only one acknowledgement, use \ack (or \Ack).
\noindent We wish to thank the Editors and two referees for helpful comments and suggestions.  We also thank Alfred  M\"{u}ller for suggesting sport examples and references, and for a thoughtful discussion of the paper, and Fabio Spizzichino for many useful comments.

\fund % Place any funding information for this work after the \fund (or \Fund) command.
\noindent The research of YM was supported by grant no. 2020063 from the United States--Israel Binational
Science Foundation (BSF). YR was supported in part by a grant from the Center for Interdisciplinary Data Science Research at the Hebrew University (CIDR).

\competing There were no competing interests to declare that arose during the preparation or publication process of this article.

\noindent
%% Use this section to describe any competing interests to declare related to this article. If there are no competing interests to declare in this section, please say
%There were no competing interests to declare which arose during the preparation or publication process of this article.

%%\data % Place any information on data related to the work in your article after the \data (or \Data) command. Omit this %%command/section and text if it is not relevant to your article.
%%\noindent The data related to the simulations found in Section 2 can be found at...

%%\supp \noindent The supplementary material for this article can be found at http://doi.org/10.1017/[TO BE SET]. % Delete this line if there are no supplementary files related to this article. If there are supplementary files related to your article, leave the line unchanged.

%%%%%%%%%%%%Reference list%%%%%%%%%%%%%%
%
% References should be in the following form (or the BibTeX file
% apt.bst should be used):
%
% For a journal:
% Surname, Initial (year). Title of paper. {\em Journal title}
% {\bf Vol,} page--range.
%
% For a book:
% Surname, Initial (year). {\em Book title}. Publisher, Address.
%
% Note the following example of a reference list.

\end{document}